\documentclass[11pt]{amsart}

\usepackage{a4wide}

\usepackage[utf8]{inputenc}
\usepackage[T1]{fontenc}

\usepackage[hidelinks]{hyperref}
\usepackage{url}
\usepackage{booktabs}
\usepackage{amsfonts}
\usepackage{amsmath}
\usepackage{amssymb}
\usepackage{nicefrac}
\usepackage{microtype}
\usepackage{cleveref}
\usepackage{graphicx}
\usepackage{doi}
\usepackage{xcolor}
\usepackage{tikz-cd}
\usepackage{lmodern}
\usepackage{mathtools}
\usepackage[thinc]{esdiff}
\usepackage{cases}
\usepackage{empheq}
\usepackage{todonotes}
\usepackage{orcidlink}
\usepackage{tikz}
\usetikzlibrary{intersections, decorations.pathreplacing}
\usepackage{caption}

\newcommand{\C}{\mathbb{C}}

\newcommand{\F}{\mathbb{F}}

\newcommand{\mg}{\mathfrak{g}}
\newcommand{\mh}{\mathfrak{h}}

\newcommand{\QGamma}{\mathrm{Q}\Gamma}
\newcommand{\Tadm}{\mathcal{T}}

\newtheorem{thm}{Theorem}[section]

\newtheorem{lemma}[thm]{Lemma}
\newtheorem{corollary}[thm]{Corollary}
\newtheorem{prop}[thm]{Proposition}

\newtheorem*{thm*}{Theorem}

\theoremstyle{definition}
\newtheorem{definition}[thm]{Definition}
\newtheorem{remark}[thm]{Remark}
\newtheorem{ex}[thm]{Example}

\crefname{prop}{proposition}{propositions}
\Crefname{prop}{Proposition}{Propositions}

\DeclareMathOperator{\ad}{ad}

\DeclareMathOperator{\End}{End}
\DeclareMathOperator{\Der}{Der}
\DeclareMathOperator{\IDer}{IDer}
\DeclareMathOperator{\car}{char}

\DeclareMathOperator{\Lie}{Lie}
\DeclareMathOperator{\id}{id}

\DeclareMathOperator{\BiDer}{BiDer}
\DeclareMathOperator{\Aut}{Aut}
\DeclareMathOperator{\Hom}{Hom}

\usepackage[
backend=bibtex,
style=numeric-comp,
sorting=nyt,
giveninits=true,
isbn=false,
url=false,
hyperref
]{biblatex}
\title[A Lie Bracket on Biderivations]{A Lie Bracket on the Space of (Right) Biderivations of a Lie Algebra}

\author[A. Di Bartolo]{Alfonso Di Bartolo\,\orcidlink{0000-0001-5619-2644}}
\address[A.~Di Bartolo, G.~La Rosa]{Dipartimento di Matematica e Informatica, Università degli Studi di Palermo, Via Archirafi 34, 90123 Palermo, Italy}
\email{alfonso.dibartolo@unipa.it,gianmarco.larosa@unipa.it}

\author[G. La Rosa]{Gianmarco La Rosa\,\orcidlink{0000-0003-1047-5993}}


\setuptodonotes{color=blue!30}
\definecolor{math}{rgb}{0.0, 0.8, 0.6}

\begin{document}

\begin{abstract}
Derivations of a Lie algebra form a Lie algebra under the commutator bracket. Motivated by this property, we investigate Lie algebra structures on spaces of biderivations. We introduce right and left biderivations and show that their spaces admit natural Lie brackets induced by pointwise commutators of derivations. We describe these structures in terms of current Lie algebras and establish a cohomological characterisation of classical biderivations, with applications to semisimple and Heisenberg Lie algebras. We also examine the relation between the right and left brackets for symmetric and skew-symmetric biderivations. Finally, we discuss the pointwise integration of right biderivations into families of Lie group automorphisms.
\end{abstract}

\maketitle


\section{Introduction}\label{sec:introduction}

A \emph{biderivation} of an algebra $A$ is a bilinear map $B\colon A\times A\to A$ that acts as a derivation in each argument. The notion has been studied in various algebraic settings, beginning with early contributions by Maksa \cite{maksa1980remark} and Vukman \cite{Vukman1989}. For Lie algebras, biderivations were introduced in \cite{Wang2013} and have since been investigated for several classes of Lie algebras, often in connection with commuting linear maps \cite{wangyu2013,hanwangxia2016,tang2018,bresarzhao2018,Chen_Yao_Zhao_2025}.

It is well known that derivations of a Lie algebra form a Lie algebra under the commutator bracket. This raises a natural question: \emph{can a similar Lie bracket be defined on a space of biderivations?} Since bilinear maps cannot be composed as linear endomorphisms, the usual commutator construction does not apply directly. A matricial approach to this problem was considered in \cite{DIBARTOLOLAROSA2025}. Here, we propose a different approach that does not depend on a choice of basis.

We introduce \emph{right} and \emph{left biderivations}, which satisfy the derivation rule in the first or second argument, respectively, without requiring linearity in the other variable. We prove that both spaces admit natural Lie brackets and describe their structures in terms of current Lie algebras. Classical biderivations arise as the intersection of these spaces and admit a cohomological characterization through the first Chevalley–Eilenberg cohomology group. This allows us to recover the known classification in the semisimple case and determine the biderivations of the three-dimensional Heisenberg Lie algebra.

We further investigate the relation between the right and left brackets for symmetric and skew-symmetric biderivations, showing in particular that the bracket of two classical biderivations need not be bilinear. Finally, we discuss the pointwise integration of right biderivations into families of Lie group automorphisms.

The paper is organised as follows. \Cref{sec:preliminaries} recalls the necessary definitions. \Cref{sec:main_results} introduces the Lie brackets and their algebraic properties. \Cref{sec:cocycles} develops the cohomological approach to classical biderivations. \Cref{sec:symm_and_skew-symm} examines symmetric and skew-symmetric biderivations, while \Cref{sec:examples} presents some examples. Finally, \Cref{sec:integration_of_rightbider} discusses the pointwise integration of right biderivations.

\section{Preliminary Definitions}
\label{sec:preliminaries}

Unless otherwise stated, \(\mg\) denotes a Lie algebra over a field \(\F\) of characteristic zero, and the square brackets \([-,-]\) denote its Lie bracket.

A \emph{derivation} of $\mg$ is a linear map \( D \colon \mg \to \mg \) satisfying the Leibniz rule:
\[
D\left([x, y]\right) = [D(x), y] + [x, D(y)]
\quad \text{for all } x, y \in \mg.
\]
The set of all derivations of \(\mg\), denoted by \(\Der(\mg)\), is a Lie subalgebra of \(\mathfrak{gl}(\mg)\) with respect to the commutator of linear maps,
\[
[D_1, D_2] := D_1 \circ D_2 - D_2 \circ D_1.
\]
The notion of \emph{biderivation} is a two-variable generalisation of that of derivation and has appeared in various, possibly more general, contexts. In the specific case of Lie algebras, it was introduced in \cite{Wang2013}.

\begin{definition}
	Let $\mg$ be a Lie algebra over a field $\F$. A bilinear map $B\colon \mg\times \mg\rightarrow \mg$ is called \emph{biderivation} of $\mg$ if it satisfies
	\begin{align}
		& B(\left[x,y\right],z)=\left[x,B(y,z)\right]+\left[B(x,z),y\right]
		\label{eq:cond1}\\
		& B(x,\left[y,z\right])=\left[B(x,y),z\right]+\left[y,B(x,z)\right]
		\label{eq:cond2},
	\end{align}
	for all $x,y,z\in \mg$.
\end{definition}

A \emph{left Leibniz algebra $L$ over a field $\F$} is a vector space over $\F$, equipped with a $\F$-bilinear map $\left[-,-\right]\colon L\times L\rightarrow L$ satisfying the Leibniz identity
\begin{equation}\label{eq:left_diff_id}
	\left[x,\left[y,z\right]\right]=\left[\left[x,y\right],z\right]+\left[y,\left[x,z\right]\right],
\end{equation}
for all $x,y,z\in L$; it is a \emph{right Leibniz algebra} if
\begin{equation}\label{eq:right_diff_id}
	\left[\left[x,y\right],z\right]=\left[\left[x,z\right],y\right]+\left[x,\left[y,z\right]\right],
\end{equation}
for all $x,y,z\in L$. Following A.~Blokh~\cite{blokh1965}, the identity~\eqref{eq:left_diff_id} is referred to as the \emph{left differential identity}, and~\eqref{eq:right_diff_id} as the \emph{right differential identity}. A Leibniz algebra that is both left and right is called \emph{symmetric}. Every statement about left Leibniz algebras has a right counterpart: given a left Leibniz algebra $\left(L,\left[\cdot,\cdot\right]\right)$, the opposite product $\left\{x,y\right\}=\left[y,x\right]$ turns $L$ into a right Leibniz algebra. For a systematic treatment we refer to \cite{ayupov2019}.

We briefly recall the degree-one part of Lie algebra cohomology, which will be used later to describe classical biderivations cohomologically.

A \emph{\(\mg\)-module} is a vector space \(M\) over \(\F\) together with a bilinear map \(\mg\times M\to M\), \((x,m)\mapsto x\cdot m\), such that
\[
[x,y]\cdot m = x\cdot(y\cdot m) - y\cdot(x\cdot m)
\quad\text{for all } x,y\in\mg,\ m\in M.
\]
A linear map \(c\colon\mg\to M\) is a \emph{\(1\)-cocycle} with values in \(M\) if
\begin{equation}\label{eq:cocycle_condition}
c([x,y]) = x\cdot c(y) - y\cdot c(x)
\quad\text{for all } x,y\in\mg.
\end{equation}
The \(1\)-cocycles form a vector space, denoted by \(Z^1(\mg,M)\). For every \(m\in M\), the linear map
\[
d_m\colon\mg\to M,
\qquad
d_m(x)=x\cdot m,
\]
is a \(1\)-cocycle. The cocycles of this form are called \emph{\(1\)-coboundaries}, and they form the subspace \(B^1(\mg,M)\subseteq Z^1(\mg,M)\). The \emph{first cohomology group} of \(\mg\) with coefficients in \(M\) is the quotient
\[
H^1(\mg,M)=Z^1(\mg,M)/B^1(\mg,M).
\]
These spaces are the degree-one part of the Chevalley--Eilenberg complex introduced in \cite{Chevalley1948}; for a modern account we refer to \cite[Chapter~7]{Weibel1994}.

We shall use the following classical vanishing result.

\begin{thm}[Whitehead's first lemma, see {\cite[Ch.~III]{Jacobson1962}}]\label{thm:whitehead}
Let \(\mg\) be a finite-dimensional semisimple Lie algebra over a field of characteristic zero, and let \(M\) be a finite-dimensional \(\mg\)-module. Then \(H^1(\mg,M)=0\).
\end{thm}

Finally, we recall the notion of centroid. The \emph{centroid} of \(\mg\) is
\[
\Gamma(\mg)
=
\left\{
T\in\End(\mg)
\ \middle|\
T([x,y])=[T(x),y]
\ \text{for all } x,y\in\mg
\right\}.
\]
By the skew-symmetry of the bracket, every \(T\in\Gamma(\mg)\) also satisfies \(T([x,y])=[x,T(y)]\). Clearly \(\F\id\subseteq\Gamma(\mg)\). If \(\mg\) is finite-dimensional and simple, then \(\Gamma(\mg)\) is a finite field extension of \(\F\). The algebra \(\mg\) is called \emph{central simple} when \(\Gamma(\mg)=\F\id\), and every simple Lie algebra over an algebraically closed field is central simple. We refer to \cite[Ch.~X]{Jacobson1962} (see also \cite{Benkart2006} for a systematic study of centroids).

\section{Main Results}\label{sec:main_results}

In this section we introduce the one-sided versions of biderivations and show that they naturally carry Lie algebra structures. The key observation is that, once one fixes the variable which is not differentiated, a one-sided biderivation is a derivation-valued map; this makes the construction of the bracket transparent.

For two vector spaces \(V\) and \(W\), we denote by \(\mathcal F(V,W)\) the vector space of all maps from \(V\) to \(W\), with pointwise operations. No linearity is assumed for the elements of \(\mathcal F(V,W)\).

We keep the terminology ``right'' and ``left'' in analogy with the right and left Leibniz identities recalled in \Cref{sec:preliminaries}. We stress that, in the one-sided setting, no linearity is imposed on the parameter variable unless explicitly stated.

\begin{definition}\label{def:rightbiderivation}
A map \(B\colon \mg\times\mg\to\mg\) is called a \emph{right biderivation} of \(\mg\) if, for every fixed \(y\in\mg\), the map \(B(\cdot,y)\colon\mg\to\mg\) is a derivation of \(\mg\).
\end{definition}

Equivalently, \(B\) is a right biderivation if it is linear in the first variable and satisfies
\begin{equation}\label{eq:rightbider}
    B([x_1,x_2],y)=[x_1,B(x_2,y)]+[B(x_1,y),x_2]
\end{equation}
for all \(x_1,x_2,y\in\mg\). We denote by \(\BiDer_r(\mg)\) the set of right biderivations of \(\mg\).

\begin{definition}\label{def:leftbiderivation}
A map \(B\colon \mg\times\mg\to\mg\) is called a \emph{left biderivation} of \(\mg\) if, for every fixed \(x\in\mg\), the map \(B(x,\cdot)\colon\mg\to\mg\) is a derivation of \(\mg\).
\end{definition}

Equivalently, \(B\) is a left biderivation if it is linear in the second variable and satisfies
\begin{equation}\label{eq:leftbiderivation}
    B(x,[y_1,y_2])=[B(x,y_1),y_2]+[y_1,B(x,y_2)]
\end{equation}
for all \(x,y_1,y_2\in\mg\). We denote by \(\BiDer_l(\mg)\) the set of left biderivations of \(\mg\).

With these definitions, the classical biderivations are precisely the maps which are both right and left biderivations: \(\BiDer(\mg)=\BiDer_r(\mg)\cap\BiDer_l(\mg)\). In particular, the intersection consists of maps which are linear in both variables and satisfy the derivation rule in each variable.

The next observation is the structural reason behind the Lie brackets introduced below.

\begin{prop}\label{prop:right_der_valued}
There is a natural vector space isomorphism
\[
\Phi_r\colon \BiDer_r(\mg)\longrightarrow \mathcal F(\mg,\Der(\mg))
\]
defined by \(\Phi_r(B)(y)(x)=B(x,y)\).
\end{prop}

\begin{proof}
If \(B\in\BiDer_r(\mg)\), then \(B(\cdot,y)\in\Der(\mg)\) for every \(y\in\mg\), so \(\Phi_r(B)\) is a well-defined element of \(\mathcal F(\mg,\Der(\mg))\). Conversely, given \(F\in\mathcal F(\mg,\Der(\mg))\), the map \(B_F(x,y)=F(y)(x)\) is a right biderivation, because each \(F(y)\) is a derivation. The two constructions are mutually inverse and linear.
\end{proof}

Similarly, \(\Phi_l(B)(x)(y)=B(x,y)\) defines a natural vector space isomorphism \(\Phi_l\colon \BiDer_l(\mg)\to \mathcal F(\mg,\Der(\mg))\). Thus one-sided biderivations are best understood as derivation-valued maps.

We now define the bracket on right biderivations.

\begin{definition}\label{def:right_bracket}
For \(B_1,B_2\in\BiDer_r(\mg)\), define \(B_1\rhd B_2\colon\mg\times\mg\to\mg\) by
\begin{equation}\label{eq:Liebracket_right}
(B_1\rhd B_2)(x,y)
=
B_1(B_2(x,y),y)-B_2(B_1(x,y),y).
\end{equation}
\end{definition}

For every \(y\in\mg\), this is nothing but the commutator of the derivations \(B_1(\cdot,y)\) and \(B_2(\cdot,y)\):
\[
\Phi_r(B_1\rhd B_2)(y)
=
[\Phi_r(B_1)(y),\Phi_r(B_2)(y)].
\]

\begin{thm}\label{thm:right_bider_lie}
The pair \((\BiDer_r(\mg),\rhd)\) is a Lie algebra over \(\F\).
\end{thm}

\begin{proof}
Through the isomorphism \(\Phi_r\), the operation \(\rhd\) becomes the pointwise commutator on \(\mathcal F(\mg,\Der(\mg))\). Since the commutator of two derivations is again a derivation, \(B_1\rhd B_2\in\BiDer_r(\mg)\); bilinearity, skew-symmetry and the Jacobi identity hold pointwise because they hold in \(\Der(\mg)\).
\end{proof}

The construction for left biderivations is analogous: for \(B_1,B_2\in\BiDer_l(\mg)\), set
\begin{equation}\label{eq:Liebracket_left}
(B_1\lhd B_2)(x,y)
=
B_1(x,B_2(x,y))-B_2(x,B_1(x,y)),
\end{equation}
so that \(\Phi_l(B_1\lhd B_2)(x)=[\Phi_l(B_1)(x),\Phi_l(B_2)(x)]\) for every \(x\in\mg\).

\begin{thm}\label{thm:left_bider_lie}
The pair \((\BiDer_l(\mg),\lhd)\) is a Lie algebra over \(\F\).
\end{thm}

\begin{proof}
Identical to the proof of \Cref{thm:right_bider_lie}, using \(\Phi_l\) in place of \(\Phi_r\).
\end{proof}

\subsection{A current Lie algebra viewpoint}

Recall a standard construction. Let \(A\) be a commutative associative \(\F\)-algebra and let \(\mathfrak a\) be a Lie algebra over \(\F\). The \emph{current Lie algebra} associated with \(A\) and \(\mathfrak a\) is the vector space \(A\otimes \mathfrak a\), endowed with the Lie bracket
\[
[f\otimes X,h\otimes Y]=fh\otimes [X,Y],
\]
for all \(f,h\in A\) and \(X,Y\in\mathfrak a\). This construction includes the usual loop and current algebras of the theory of infinite-dimensional Lie algebras (see \cite{kac1990infinite} and \cite{neher_savage_senesi2012} for the related viewpoint of Lie algebras of maps).

In our setting, let \(A=\mathcal F(\mg,\F)\) be the commutative associative algebra of all \(\F\)-valued functions on \(\mg\), with pointwise multiplication. \Cref{prop:right_der_valued} shows that right biderivations may be viewed as derivation-valued currents on \(\mg\). When \(\Der(\mg)\) is finite-dimensional, this can be written in tensorial form.

\begin{prop}\label{prop:right_bider_current_algebra}
Assume that \(\Der(\mg)\) is finite-dimensional (as is the case, for instance, whenever \(\mg\) is finite-dimensional). Then:
\begin{itemize}
    \item[\textup{(i)}] The Lie algebra \((\BiDer_r(\mg),\rhd)\) is naturally isomorphic to the current Lie algebra \(A\otimes \Der(\mg)\).
    \item[\textup{(ii)}] \(Z(\BiDer_r(\mg))\simeq A\otimes Z(\Der(\mg))\)
    \item[\textup{(iii)}] \([\BiDer_r(\mg),\BiDer_r(\mg)]\simeq A\otimes [\Der(\mg),\Der(\mg)]\) and, more generally, \(\BiDer_r(\mg)^{(k)}\simeq A\otimes \Der(\mg)^{(k)}\) for every \(k\geq 0\).
    \item[\textup{(iv)}] \(\gamma_k(\BiDer_r(\mg))\simeq A\otimes \gamma_k(\Der(\mg))\) for every \(k\geq 1\).
    \item[\textup{(v)}] \(\BiDer_r(\mg)\) is solvable (resp.\ nilpotent) if and only if \(\Der(\mg)\) is solvable (resp.\ nilpotent).
\end{itemize}
\end{prop}

\begin{proof}
(i) By \Cref{prop:right_der_valued}, \(\Phi_r\) identifies \((\BiDer_r(\mg),\rhd)\) with \(\mathcal F(\mg,\Der(\mg))\) equipped with the pointwise commutator, so it suffices to identify the latter with \(A\otimes\Der(\mg)\). The bilinear map \((f,D)\mapsto\bigl(y\mapsto f(y)D\bigr)\) induces a linear map
\[
\Theta\colon A\otimes\Der(\mg)\longrightarrow\mathcal F(\mg,\Der(\mg)),
\qquad
\Theta(f\otimes D)(y)=f(y)D.
\]
Fix a basis \(D_1,\dots,D_N\) of \(\Der(\mg)\). Every \(F\in\mathcal F(\mg,\Der(\mg))\) can be written uniquely as \(F(y)=\sum_i f_i(y)D_i\) with \(f_i\in A\), whence \(F=\Theta\bigl(\sum_i f_i\otimes D_i\bigr)\); the same uniqueness, applied pointwise, shows that \(\Theta\) is injective. Finally,
\[
\Theta\bigl([f\otimes D,h\otimes E]\bigr)(y)
=
f(y)h(y)[D,E]
=
[f(y)D,h(y)E],
\]
so \(\Theta\) intertwines the current bracket with the pointwise commutator, and is therefore an isomorphism of Lie algebras.

(ii) Under the above identification, \(B\) is central if and only if \(B(\cdot,y)\) commutes with \(C(\cdot,y)\) for every \(C\in\BiDer_r(\mg)\) and every \(y\in\mg\). Taking for \(C\) the constant maps \(y\mapsto D\), \(D\in\Der(\mg)\), we see that this happens if and only if \(B(\cdot,y)\in Z(\Der(\mg))\) for every \(y\). Hence \(Z(\BiDer_r(\mg))\simeq\mathcal F(\mg,Z(\Der(\mg)))\simeq A\otimes Z(\Der(\mg))\).

(iii) Every pointwise commutator takes values in \([\Der(\mg),\Der(\mg)]\), so \([\BiDer_r(\mg),\BiDer_r(\mg)]\subseteq\mathcal F(\mg,[\Der(\mg),\Der(\mg)])\). Conversely, if \(f\in A\) and \(D=[D_1,D_2]\), then \(y\mapsto f(y)D\) is the commutator of \(y\mapsto f(y)D_1\) with the constant map \(y\mapsto D_2\). The identity for the whole derived series follows by induction, and (iv) is proved in the same way. Statement (v) is an immediate consequence of (iii) and (iv).
\end{proof}

The current algebra description shows that the structure of \(\BiDer_r(\mg)\) is controlled by \(\Der(\mg)\), rather than directly by \(\mg\). Hence one cannot, in general, transfer properties from \(\mg\) to \(\BiDer_r(\mg)\) without first understanding the derivation algebra. In several important cases, however, \(\Der(\mg)\) is completely determined by \(\mg\). For example, this happens when \(\mg\) is complete and in particular when \(\mg\) is semisimple over a field of characteristic zero.

\begin{corollary}\label{cor:complete_current_consequences}
Assume that \(\mg\) is finite-dimensional and complete, that is, $Z(\mg)=0$ and $\Der(\mg)=\IDer(\mg)$. Then there is a natural Lie algebra isomorphism
\[
\BiDer_r(\mg)\simeq \mathcal F(\mg,\F)\otimes \mg .
\]
Moreover, $Z(\BiDer_r(\mg))=0$ and $[\BiDer_r(\mg),\BiDer_r(\mg)]\simeq\mathcal F(\mg,\F)\otimes[\mg,\mg]$. In particular, if \(\mg\) is semisimple, then \(\BiDer_r(\mg)\) is perfect.
\end{corollary}

\begin{proof}
Since \(\mg\) is complete, \(\ad\colon\mg\to\Der(\mg)\) is an isomorphism of Lie algebras, and the first claim follows from \Cref{prop:right_bider_current_algebra}(i). Under this identification, an element of \(\BiDer_r(\mg)\) is a map \(y\mapsto X(y)\in\mg\), which is central if and only if \([X(y),Y]=0\) for all \(y,Y\in\mg\); since \(Z(\mg)=0\), this forces \(X\equiv 0\). The description of the derived algebra is \Cref{prop:right_bider_current_algebra}(iii), and if \(\mg\) is semisimple then \([\mg,\mg]=\mg\), so \(\BiDer_r(\mg)\) is perfect.
\end{proof}

\begin{remark}
Even when \(\mg\) is simple, the Lie algebra \(\BiDer_r(\mg)\) is generally not simple: under the identification \(\BiDer_r(\mg)\simeq \mathcal F(\mg,\F)\otimes\mg\), every ideal \(I\subseteq \mathcal F(\mg,\F)\) gives rise to an ideal \(I\otimes\mg\) of \(\BiDer_r(\mg)\).
\end{remark}

\subsection{Classical biderivations inside the linear sector}
\label{subsec:classical_biderivations_linear_sector}

The identification \(\Phi_r\) allows us to locate the classical biderivations inside the larger space of right biderivations.

Let \(B\in\BiDer_r(\mg)\), and let
\[
\beta\colon \mg\longrightarrow \Der(\mg),
\qquad
\beta(y)(x)=B(x,y),
\]
be the corresponding derivation-valued map. Then \(B\) is bilinear if and only if \(\beta\) is linear: linearity of \(B\) in the first argument is part of the definition of right biderivation, while linearity in the second argument is precisely the linearity of \(\beta\). Hence the space of bilinear right biderivations is naturally identified with \(\Hom(\mg,\Der(\mg))\), the linear sector of \(\mathcal F(\mg,\Der(\mg))\).

Classical biderivations are obtained by imposing one further compatibility condition on this linear sector.

\begin{definition}\label{def:classical_linear_sector}
We set
\[
\mathcal B(\mg)
=
\left\{
\beta\in\Hom(\mg,\Der(\mg))
\ \middle|\
\beta([y,z])
=
\ad_y\circ\beta(z)-\ad_z\circ\beta(y)
\ \text{for all } y,z\in\mg
\right\}.
\]
\end{definition}

The condition in \Cref{def:classical_linear_sector} is an identity in \(\End(\mg)\); evaluated at \(x\in\mg\), it reads \(\beta([y,z])(x)=[y,\beta(z)(x)]-[z,\beta(y)(x)]\).

\begin{prop}\label{prop:classical_biderivations_linear_sector}
The assignment \(\beta(y)(x)=B(x,y)\) defines a natural vector space isomorphism \(\BiDer(\mg)\simeq \mathcal B(\mg)\).
\end{prop}

\begin{proof}
Let \(B\in\BiDer(\mg)\) and let \(\beta\) be the associated linear map. Since \(B\) is a left biderivation, for all \(x,y,z\in\mg\),
\[
\beta([y,z])(x)=B(x,[y,z])=[B(x,y),z]+[y,B(x,z)]
=[y,\beta(z)(x)]-[z,\beta(y)(x)],
\]
where we used the skew-symmetry of the bracket of \(\mg\) in the last step. Hence \(\beta\in\mathcal B(\mg)\).

Conversely, let \(\beta\in\mathcal B(\mg)\) and set \(B_\beta(x,y)=\beta(y)(x)\). Since each \(\beta(y)\) is a derivation, \(B_\beta\) is a bilinear right biderivation, and reversing the computation above shows that the defining condition of \(\mathcal B(\mg)\) is exactly the left biderivation identity for \(B_\beta\). The two constructions are mutually inverse.
\end{proof}

\begin{remark}\label{rem:classical_biderivations_linear_sector}
\Cref{prop:classical_biderivations_linear_sector} yields the chain of natural inclusions
\[
\BiDer(\mg)
\simeq
\mathcal B(\mg)
\subseteq
\Hom(\mg,\Der(\mg))
\subseteq
\mathcal F(\mg,\Der(\mg))
\simeq
\BiDer_r(\mg).
\]
Classical biderivations are thus the linear derivation-valued currents satisfying the compatibility condition of \Cref{def:classical_linear_sector}.
\end{remark}

\begin{remark}\label{rem:linear_sector_not_closed}
The linear sector \(\Hom(\mg,\Der(\mg))\), and hence also the subspace \(\mathcal B(\mg)\), is not in general preserved by the bracket \(\rhd\). Indeed, if \(B_1,B_2\) are bilinear right biderivations with associated linear maps \(\beta_1,\beta_2\), then
\[
\Phi_r(B_1\rhd B_2)(y)
=
[\beta_1(y),\beta_2(y)],
\]
and the right-hand side is generally quadratic in \(y\): writing \(\beta_1(y)=\sum_i \lambda_i(y)D_i\) and \(\beta_2(y)=\sum_j \mu_j(y)E_j\) with \(\lambda_i,\mu_j\in\mg^*\) and \(D_i,E_j\in\Der(\mg)\), one has
\[
[\beta_1(y),\beta_2(y)]
=
\sum_{i,j}\lambda_i(y)\mu_j(y)[D_i,E_j].
\]
Therefore \(B_1\rhd B_2\) need not be bilinear, and hence need not be a classical biderivation. \Cref{ex:heisenberg_biderright} below gives an explicit instance of this phenomenon.
\end{remark}

\begin{remark}[Relation with the matricial approach]\label{rem:relation_matricial_approach}
Assume that \(\mg\) is finite-dimensional, and fix a basis \(\{e_1,\dots,e_n\}\) of \(\mg\). A bilinear right biderivation \(B\), with associated linear map \(\beta\colon\mg\to\Der(\mg)\), is determined by the \(n\) derivations
\[
D_i=\beta(e_i)=B(\cdot,e_i),
\qquad 1\leq i\leq n.
\]
Representing each \(D_i\) by a matrix with respect to the chosen basis, \(B\) corresponds to an \(n\)-tuple
\[
(D_1,\dots,D_n)\in \Der(\mg)^n\subseteq M_n(\F)^n,
\]
which is precisely the matricial description used in \cite{DIBARTOLOLAROSA2025}. In this language, the condition for \(B\) to be a classical biderivation is the condition of \Cref{def:classical_linear_sector} on \(\beta\), which in coordinates becomes a system of linear relations among the matrices \(D_1,\dots,D_n\). The current algebra viewpoint also explains why the bracket \(\rhd\) is natural on the one-sided space but does not preserve the matricial/bilinear sector: as observed in \Cref{rem:linear_sector_not_closed}, for \(y=\sum_i y_ie_i\) one has \([\beta_1(y),\beta_2(y)]=\sum_{i,j}y_iy_j[D_i,E_j]\), which is quadratic in \(y\).
\end{remark}

The right and left constructions are related by transposition. For \(B\colon\mg\times\mg\to\mg\), write \(B^t(x,y)=B(y,x)\).

\begin{prop}\label{prop:transpose_right_left}
The transpose map gives a vector space isomorphism \(\BiDer_r(\mg)\to\BiDer_l(\mg)\). Moreover, for all \(B_1,B_2\in\BiDer_r(\mg)\),
\[
(B_1\rhd B_2)^t=B_1^t\lhd B_2^t.
\]
In particular, the right and left Lie algebra structures are naturally isomorphic.
\end{prop}

\begin{proof}
If \(B\in\BiDer_r(\mg)\), then for fixed \(x\), the map \(y\mapsto B^t(x,y)=B(y,x)\) is a derivation, so \(B^t\in\BiDer_l(\mg)\); applying the same argument twice shows that transposition is an isomorphism. As for the brackets,
\begin{align*}
(B_1^t\lhd B_2^t)(y,x)
&=
B_1^t(y,B_2^t(y,x))-B_2^t(y,B_1^t(y,x))\\
&=
B_1(B_2(x,y),y)-B_2(B_1(x,y),y)
=
(B_1\rhd B_2)(x,y),
\end{align*}
which is exactly \((B_1\rhd B_2)^t=B_1^t\lhd B_2^t\).
\end{proof}

The following example shows that, even when \(B_1\) and \(B_2\) are classical biderivations, the bracket \(B_1\rhd B_2\) need not be a classical biderivation. We also find the example, and the related computations, useful in themselves: to the best of our knowledge, the existing literature on biderivations contains few explicit examples beyond those that are inner or induced by linear commuting maps.

\begin{ex}\label{ex:heisenberg_biderright}
	Let $\mh$ be the $3$-dimensional Heisenberg Lie algebra with basis $\left\{e_1,e_2,e_3\right\}$ and non-zero bracket $[e_1,e_2]=e_3$. The following bilinear maps are biderivations of $\mh$ (the values on pairs of basis elements not listed are zero):
    \begin{gather*}
        B_1(e_1,e_1)=-B_1(e_2,e_2)=e_3,\,B_1(e_1,e_2)=B_1(e_2,e_1)=e_1,\\
        B_1(e_2,e_3)=B_1(e_3,e_2)=e_3, \\
        \\B_2(e_1,e_1)=e_1,\,B_2(e_2,e_2)=e_2,\\ B_2(e_1,e_3)=B_2(e_3,e_1)=B_2(e_2,e_3)=B_2(e_3,e_2)=e_3.
    \end{gather*}

Both \( B_1 \) and \( B_2 \) are symmetric bilinear maps; hence, by \Cref{lem:rsymm_l} below, it suffices to check that \Cref{eq:rightbider} is satisfied.
Let $x=\sum_{i=1}^3x_ie_i,y=\sum_{i=1}^3y_ie_i,z=\sum_{i=1}^3z_ie_i\in\mh$, so that $[x,y]=\left(x_1y_2-x_2y_1\right)e_3$. On one hand,
\begin{equation*}
    B_1\left([x,y],z\right)=B_1\left(\left(x_1y_2-x_2y_1\right)e_3,z\right)=z_2\left(x_1y_2-x_2y_1\right)e_3.
\end{equation*}
On the other hand,
\begin{align*}
   [x,B_1(y,z)]+[B_1(x,z),y]&=[x,\left(y_1z_2+y_2z_1\right)e_1+\left(y_1z_1-y_2z_2+y_2z_3+y_3z_2\right)e_3]\\
   &\quad+[\left(x_1z_2+x_2z_1\right)e_1+\left(x_1z_1-x_2z_2+x_2z_3+x_3z_2\right)e_3,y]\\
   &=-x_2\left(y_1z_2+y_2z_1\right)e_3+y_2\left(x_1z_2+x_2z_1\right)e_3\\
   &=\left(-x_2y_1z_2+x_1y_2z_2\right)e_3.
\end{align*}
Hence $B_1\in\BiDer(\mh)$. Regarding \( B_2 \), on the one hand,
\begin{equation*}
   B_2\left([x,y],z\right)=\left(x_1y_2-x_2y_1\right)\left(z_1+z_2\right)e_3,
\end{equation*}
while on the other hand
\begin{align*}
    [x,B_2(y,z)]+[B_2(x,z),y]&=[x,y_1z_1e_1+y_2z_2e_2+\left(y_1z_3+y_2z_3+y_3z_1+y_3z_2\right)e_3]\\
    &\quad +[x_1z_1e_1+x_2z_2e_2+\left(x_1z_3+x_2z_3+x_3z_1+x_3z_2\right)e_3,y]\\
    &=\left(x_1y_2z_2-x_2y_1z_1\right)e_3+\left(x_1y_2z_1-x_2y_1z_2\right)e_3\\
    &=\left(x_1y_2-x_2y_1\right)(z_1+z_2)e_3,
\end{align*}
and then $B_2\in\BiDer(\mh)$.

We now compute \( B := B_1 \rhd B_2 \). A direct calculation from \eqref{eq:Liebracket_right} gives
\begin{equation}\label{eq:heisenberg_bracket_formula}
    B(x,y)
    =
    x_2y_1\left(y_2-y_1\right)e_1
    +
    \left(x_2y_1y_2-x_1y_1y_2-2x_2y_1y_3+x_2y_2y_3\right)e_3.
\end{equation}
By \Cref{thm:right_bider_lie}, \(B\) is a right biderivation. However, as predicted by \Cref{rem:linear_sector_not_closed}, \(B\) is quadratic in the second variable: for instance, the only non-zero value of \(B\) on pairs of basis elements is \(B(e_2,e_1)=-e_1\), whereas
\[
B(e_2,e_1+e_2)=e_3\neq -e_1=B(e_2,e_1)+B(e_2,e_2).
\]
In particular, \(B\) is not bilinear, and hence not a classical biderivation. One can also verify directly that the left biderivation identity fails. Since \([y,z]\in\F e_3\) and, by \eqref{eq:heisenberg_bracket_formula}, \(B(x,w)=0\) whenever \(w\in\F e_3\), we have \(B(x,[y,z])=0\) for all \(x,y,z\in\mh\); on the other hand,
\[
[B(x,y),z]+[y,B(x,z)]
=
x_2\bigl(y_1z_2(y_2-y_1)-y_2z_1(z_2-z_1)\bigr)e_3,
\]
which is non-zero, e.g., for \(x=e_2\), \(y=e_1+2e_2\), \(z=e_2\).
\end{ex}

\section{A Cohomological Description of Classical Biderivations}
\label{sec:cocycles}

In this section we show that the compatibility condition of
\Cref{def:classical_linear_sector} is a cocycle condition for a natural
\(\mg\)-module structure on \(\End(\mg)\). This observation packages the space
of classical biderivations into an exact sequence involving the first
Chevalley--Eilenberg cohomology group, and reduces their classification to the
computation of a cohomology group and of a distinguished space of
endomorphisms of \(\mg\). As applications, we recover in a uniform,
coordinate-free way the known description of the biderivations of a semisimple
Lie algebra, and we compute all classical biderivations of the
three-dimensional Heisenberg algebra. Throughout, we freely use the notation of the preliminaries (\Cref{sec:preliminaries}).

\subsection{Biderivations as cocycles}

We endow \(\End(\mg)\) with the \(\mg\)-module structure $x\cdot T=\ad_x\circ T$.
Indeed,
\[
[x,y]\cdot T
=
\ad_{[x,y]}\circ T
=
\ad_x\circ\ad_y\circ T-\ad_y\circ\ad_x\circ T
=
x\cdot(y\cdot T)-y\cdot(x\cdot T).
\]

\begin{prop}\label{prop:B_is_cocycles}
With respect to the $\mg$-module structure of $\End(\mg)$,
\[
\mathcal B(\mg)
=
Z^1\bigl(\mg,\End(\mg)\bigr)\cap\Hom\bigl(\mg,\Der(\mg)\bigr).
\]
Consequently, by \Cref{prop:classical_biderivations_linear_sector}, there is a
natural vector space isomorphism
\[
\BiDer(\mg)
\simeq
Z^1\bigl(\mg,\End(\mg)\bigr)\cap\Hom\bigl(\mg,\Der(\mg)\bigr),
\]
assigning to a classical biderivation \(B\) the cocycle
\(\beta(y)=B(\cdot,y)\).
\end{prop}

\begin{proof}
For a linear map \(\beta\colon\mg\to\End(\mg)\), the cocycle condition
\eqref{eq:cocycle_condition} reads
\[
\beta([y,z])
=
y\cdot\beta(z)-z\cdot\beta(y)
=
\ad_y\circ\beta(z)-\ad_z\circ\beta(y),
\]
which is exactly the defining condition of \(\mathcal B(\mg)\) in
\Cref{def:classical_linear_sector}. Hence \(\mathcal B(\mg)\) consists
precisely of the cocycles taking values in \(\Der(\mg)\).
\end{proof}

We now examine the coboundaries. For \(T\in\End(\mg)\), the associated
coboundary is
\[
d_T\colon\mg\longrightarrow\End(\mg),
\qquad
d_T(y)=\ad_y\circ T,
\]
and it corresponds, under the identification of
\Cref{prop:classical_biderivations_linear_sector}, to the bilinear map
\begin{equation}\label{eq:BT}
B_T(x,y)=[y,T(x)].
\end{equation}
Every coboundary is a cocycle, but \(d_T\) need not take values in
\(\Der(\mg)\): this leads to the following definition.

\begin{definition}\label{def:admissible}
An endomorphism \(T\in\End(\mg)\) is called \emph{admissible} if
\(\ad_y\circ T\in\Der(\mg)\) for every \(y\in\mg\). We denote by
\(\Tadm(\mg)\) the space of admissible endomorphisms of \(\mg\). We also
recall the \emph{quasicentroid} of \(\mg\) (see \cite{Leger2000}),
\[
\QGamma(\mg)
=
\left\{
T\in\End(\mg)
\ \middle|\
[T(x),y]=[x,T(y)]
\ \text{for all } x,y\in\mg
\right\}.
\]
\end{definition}

\begin{lemma}\label{lem:admissible_inclusions}
The following inclusions hold:
\[
\Gamma(\mg)\subseteq\QGamma(\mg)\subseteq\Tadm(\mg)
\supseteq
\Hom(\mg,Z(\mg)),
\]
where \(\Hom(\mg,Z(\mg))\) is regarded as the subspace of endomorphisms of
\(\mg\) with image contained in \(Z(\mg)\). Moreover, for
\(T\in\Tadm(\mg)\), the map \(B_T\) of \eqref{eq:BT} is a classical
biderivation of \(\mg\), and \(B_T=0\) if and only if
\(T\in\Hom(\mg,Z(\mg))\).
\end{lemma}

\begin{proof}
If \(T\in\Gamma(\mg)\), then \(T([x,y])=[T(x),y]=[x,T(y)]\), and the last
equality gives \(T\in\QGamma(\mg)\). If \(T\in\QGamma(\mg)\), then
\[
(\ad_y\circ T)(x)=[y,T(x)]=-[T(x),y]=-[x,T(y)]=[T(y),x]=\ad_{T(y)}(x),
\]
so \(\ad_y\circ T=\ad_{T(y)}\in\IDer(\mg)\subseteq\Der(\mg)\), and
\(T\in\Tadm(\mg)\). If \(T\) has image in \(Z(\mg)\), then
\(\ad_y\circ T=0\in\Der(\mg)\).

Let now \(T\in\Tadm(\mg)\). Then \(d_T\) is a coboundary, hence a cocycle,
taking values in \(\Der(\mg)\); by \Cref{prop:B_is_cocycles}, the map \(B_T\)
is a classical biderivation. Finally, \(B_T=0\) means \([y,T(x)]=0\) for all
\(x,y\in\mg\), that is, \(T(x)\in Z(\mg)\) for every \(x\).
\end{proof}

\begin{remark}\label{rem:inner_biderivations}
For \(T=\lambda\id\), \(\lambda\in\F\), formula \eqref{eq:BT} gives
\(B_T(x,y)=\lambda[y,x]\). Up to the sign of the scalar, these are the
\emph{inner biderivations} \((x,y)\mapsto\lambda[x,y]\). More generally, for \(T\in\Gamma(\mg)\) one has
\(B_T(x,y)=[y,T(x)]=T([y,x])\), so admissible endomorphisms lying in the
centroid produce the biderivations induced by centroid elements.
\end{remark}

The following statement organises the discussion so far.

\begin{thm}\label{thm:exact_sequence}
The sequence of vector spaces
\[
0
\longrightarrow
\Hom(\mg,Z(\mg))
\stackrel{\iota}{\longrightarrow}
\Tadm(\mg)
\stackrel{d}{\longrightarrow}
\BiDer(\mg)
\stackrel{\mathrm{cl}}{\longrightarrow}
H^1\bigl(\mg,\End(\mg)\bigr)
\]
is exact, where \(\iota\) is the inclusion, \(d(T)=B_T\), and
\(\mathrm{cl}\) assigns to a classical biderivation \(B\) the cohomology
class of the cocycle \(\beta(y)=B(\cdot,y)\).
\end{thm}

\begin{proof}
Injectivity of \(\iota\) is clear, and exactness at \(\Tadm(\mg)\) is the
last assertion of \Cref{lem:admissible_inclusions}. For exactness at
\(\BiDer(\mg)\): if \(B=B_T\), then its cocycle is the coboundary \(d_T\),
so \(\mathrm{cl}(B)=0\). Conversely, if \(\mathrm{cl}(B)=0\), then the
cocycle \(\beta\) of \(B\) equals \(d_T\) for some \(T\in\End(\mg)\). Since
\(\beta\) takes values in \(\Der(\mg)\), the endomorphism \(T\) is
admissible, and \(B=B_T\) lies in the image of \(d\).
\end{proof}

\begin{remark}\label{rem:H1_obstruction}
\Cref{thm:exact_sequence} splits the classification of classical
biderivations into two problems of a different nature: the computation of the
space \(\Tadm(\mg)\) of admissible endomorphisms, which produces the
biderivations of the ``inner type'' \eqref{eq:BT}, and the computation of the
image of \(\mathrm{cl}\) inside \(H^1(\mg,\End(\mg))\), which measures how
many biderivations are not of this type. When
\(H^1(\mg,\End(\mg))=0\), every classical biderivation is of the form
\(B_T\).
\end{remark}

\subsection{The semisimple case}

Throughout this subsection, \(\mg\) is finite-dimensional and semisimple over
\(\F\), \(\car\F=0\), and \(\kappa\) denotes its Killing form. Recall that
\(\kappa\) is nondegenerate and invariant, that \(Z(\mg)=0\), and that
\(\Der(\mg)=\IDer(\mg)=\ad(\mg)\). For \(S\in\End(\mg)\), we write \(S^*\)
for the adjoint of \(S\) with respect to \(\kappa\); the invariance of
\(\kappa\) gives \(\ad_y^*=-\ad_y\) for every \(y\in\mg\).

\begin{prop}\label{prop:admissible_is_quasicentroid}
If \(\mg\) is finite-dimensional and semisimple, then
\(\Tadm(\mg)=\QGamma(\mg)\).
\end{prop}

\begin{proof}
The inclusion \(\QGamma(\mg)\subseteq\Tadm(\mg)\) holds in general by
\Cref{lem:admissible_inclusions}. Conversely, let \(T\in\Tadm(\mg)\). Since
\(\Der(\mg)=\ad(\mg)\) and \(Z(\mg)=0\), for every \(y\in\mg\) there exists a
unique \(\sigma(y)\in\mg\) such that
\begin{equation}\label{eq:sigma_def}
\ad_y\circ T=\ad_{\sigma(y)},
\end{equation}
and the map \(\sigma\colon\mg\to\mg\) is linear, because \(y\mapsto\ad_y\circ T\)
is linear and \(\ad\) is injective.

Taking adjoints with respect to \(\kappa\) in \eqref{eq:sigma_def}, and using
\(\ad_z^*=-\ad_z\), we obtain \(-T^*\circ\ad_y=-\ad_{\sigma(y)}\), that is,
\begin{equation}\label{eq:sigma_transpose}
T^*\circ\ad_y=\ad_{\sigma(y)}=\ad_y\circ T
\qquad\text{for every } y\in\mg.
\end{equation}
On the other hand, evaluating \eqref{eq:sigma_def} and using the invariance and
the symmetry of \(\kappa\), for all \(x,y,z\in\mg\) we get
\[
\kappa\bigl(x,T^*([z,y])\bigr)
=
\kappa\bigl(T(x),[z,y]\bigr)
=
\kappa\bigl(z,[y,T(x)]\bigr)
=
\kappa\bigl(z,[\sigma(y),x]\bigr)
=
\kappa\bigl(x,[z,\sigma(y)]\bigr),
\]
whence, by nondegeneracy,
\begin{equation}\label{eq:sigma_second}
T^*\circ\ad_z=\ad_z\circ\sigma
\qquad\text{for every } z\in\mg.
\end{equation}
Comparing \eqref{eq:sigma_transpose} and \eqref{eq:sigma_second} yields
\(\ad_z\circ(\sigma-T)=0\) for every \(z\in\mg\), so the image of \(\sigma-T\)
is contained in \(Z(\mg)=0\), and \(\sigma=T\). Therefore
\(\ad_y\circ T=\ad_{T(y)}\) for every \(y\), which evaluated at \(x\) reads
\([y,T(x)]=[T(y),x]\), that is, \([T(x),y]=[x,T(y)]\). Hence
\(T\in\QGamma(\mg)\).
\end{proof}

The remaining ingredient is the coincidence of the quasicentroid with the
centroid, which is a classical result of Leger and Luks.

\begin{prop}[Leger--Luks {\cite[Theorem~5.28]{Leger2000}}]\label{prop:quasicentroid_is_centroid}
Let \(\mg\) be a Lie algebra which is centerless and perfect. Then \(\QGamma(\mg)=\Gamma(\mg)\).
\end{prop}

The result applies directly over any field of characteristic zero:
indeed, \cite[Theorem~5.28]{Leger2000} states that
\(\QGamma(\mg)=\Gamma(\mg)\) whenever
\(Z(\mg)=0\) and \([\mg,\mg]=\mg\).

Combining the exact sequence with Whitehead's lemma, we recover the
classification of classical biderivations in the semisimple case. The result
itself is known (see \Cref{rem:literature_semisimple}). However, the point here is that
it follows in one stroke from the cohomological description, with the vanishing
of \(H^1\) doing all the work.

\begin{thm}\label{thm:biderivations_semisimple}
Let \(\mg\) be a finite-dimensional semisimple Lie algebra over a field \(\F\)
of characteristic zero. Then the map
\begin{align*}
    \Gamma(\mg)&\longrightarrow\BiDer(\mg)\\
    T&\longmapsto B_T,
\end{align*}
where $B_T(x,y)=[y,T(x)]=T([y,x])$,
is an isomorphism of vector spaces. In particular:
\begin{itemize}
    \item[\textup{(i)}] If \(\mg\) is central simple, then every classical biderivation of
    \(\mg\) is of the form
    \[
    B(x,y)=\lambda[x,y],
    \qquad
    \lambda\in\F,
    \]
    and \(\dim\BiDer(\mg)=1\).
    \item[\textup{(ii)}] In general, if \(\mg=\mg_1\oplus\dots\oplus\mg_k\) is
    the decomposition into simple ideals, then
    \(\Gamma(\mg)=\bigoplus_i\Gamma(\mg_i)\), each summand a finite field
    extension of \(\F\). Moreover, every classical biderivation preserves each ideal
    \(\mg_i\), vanishes on \(\mg_i\times\mg_j\) for \(i\neq j\), and acts on
    \(\mg_i\) as a centroidal multiple of the bracket.
\end{itemize}
\end{thm}

\begin{proof}
The map is well defined and injective by \Cref{lem:admissible_inclusions},
since \(\Gamma(\mg)\subseteq\Tadm(\mg)\) and \(Z(\mg)=0\). For surjectivity,
let \(B\in\BiDer(\mg)\) with associated cocycle
\(\beta\in Z^1(\mg,\End(\mg))\). Since \(\mg\) is finite-dimensional, so is
\(\End(\mg)\), and \Cref{thm:whitehead} gives \(\beta=d_T\) for some
\(T\in\End(\mg)\). As \(\beta\) takes values in \(\Der(\mg)\), the
endomorphism \(T\) is admissible, and by
\Cref{prop:admissible_is_quasicentroid,prop:quasicentroid_is_centroid} we get
\(T\in\Gamma(\mg)\). In addition \(B=B_T\). The identity \(B_T(x,y)=T([y,x])\)
for \(T\in\Gamma(\mg)\) was observed in \Cref{rem:inner_biderivations}.

For (i), if \(\Gamma(\mg)=\F\id\), then \(B=B_{\mu\id}\) for some
\(\mu\in\F\), that is, \(B(x,y)=\mu[y,x]=\lambda[x,y]\) with
\(\lambda=-\mu\). For (ii), every \(T\in\Gamma(\mg)\) preserves the simple
ideals (they are the minimal ideals of \(\mg\), and
\(T([\mg_i,\mg])\subseteq[\mg_i,\mg]\)). Hence
\(\Gamma(\mg)=\bigoplus_i\Gamma(\mg_i)\), and \(\Gamma(\mg_i)\) is a finite
field extension of \(\F\) because \(\mg_i\) is simple. Finally,
\(B_T(x,y)=[y,T(x)]\in[\mg_j,\mg_i]=0\) for \(x\in\mg_i\), \(y\in\mg_j\),
\(i\neq j\).
\end{proof}

\begin{remark}\label{rem:literature_semisimple}
\Cref{thm:biderivations_semisimple} is not new. Over \(\F=\C\), part~(i) is the
classification of the biderivations of finite-dimensional complex simple Lie
algebras obtained by Tang \cite{tang2018}, and the description of biderivations
of semisimple Lie algebras through the centroid is part of the circle of ideas
around commuting linear maps studied by Bre\v{s}ar and Zhao \cite{bresarzhao2018}.
The equality \(\QGamma(\mg)=\Gamma(\mg)\) used in the proof is itself the
classical result of Leger and Luks \cite{Leger2000} recalled in
\Cref{prop:quasicentroid_is_centroid}. Our contribution here is only
methodological: the classification is obtained uniformly from Whitehead's
lemma through the cocycle description, and the exact sequence
\Cref{thm:exact_sequence} isolates exactly what obstructs it beyond the
semisimple case, namely the non-vanishing of \(H^1(\mg,\End(\mg))\) together
with the gap between \(\Tadm(\mg)\) and \(\Gamma(\mg)\).
\end{remark}

\begin{remark}\label{rem:classical_sector_abelian_ss}
\Cref{thm:biderivations_semisimple} has a curious consequence for the bracket
\(\rhd\) of \Cref{def:right_bracket}. Let \(\mg\) be semisimple and let
\(B_S,B_T\in\BiDer(\mg)\) with \(S,T\in\Gamma(\mg)\). Since centroid elements
commute with all the maps \(\ad_y\), for all \(x,y\in\mg\) we compute
\[
(B_S\rhd B_T)(x,y)
=
[y,S([y,T(x)])]-[y,T([y,S(x)])]
=
\ad_y^2\bigl(ST-TS\bigr)(x)
=
0,
\]
because \(\Gamma(\mg)=\bigoplus_i\Gamma(\mg_i)\) is commutative. Hence, for
semisimple \(\mg\), the classical sector \(\BiDer(\mg)\) is an abelian Lie
subalgebra of \((\BiDer_r(\mg),\rhd)\). This is in sharp contrast with the
nilpotent case: \Cref{ex:heisenberg_biderright} exhibits two classical
biderivations of the Heisenberg algebra whose bracket is nonzero, and not even
bilinear.
\end{remark}

\subsection{The Heisenberg algebra revisited}

We now show how the cocycle description computes all classical biderivations
of the three-dimensional Heisenberg algebra \(\mh\), with basis
\(\{e_1,e_2,e_3\}\) and nonzero bracket \([e_1,e_2]=e_3\). Recall (see
\Cref{sec:examples}) that, with respect to this ordered basis,
\[
\Der(\mh)
=
\left\{
\begin{pmatrix}
a & b & 0\\
c & d & 0\\
e & f & a+d
\end{pmatrix}
\ \middle|\
a,b,c,d,e,f\in\F
\right\},
\]
so that \(\dim\Der(\mh)=6\), and that \(\ad_{e_1}=E_{32}\),
\(\ad_{e_2}=-E_{31}\), \(\ad_{e_3}=0\), where \(E_{ij}\) denotes the
elementary matrix mapping \(e_j\) to \(e_i\).

\begin{prop}\label{prop:bider_heisenberg}
The linear map
\[
\mathcal B(\mh)\longrightarrow\Der(\mh)\oplus\Der(\mh),
\qquad
\beta\longmapsto(\beta(e_1),\beta(e_2)),
\]
is an isomorphism of vector spaces, with inverse given by
\[
(D_1,D_2)
\longmapsto
\beta,
\qquad
\beta(e_1)=D_1,\quad
\beta(e_2)=D_2,\quad
\beta(e_3)=\ad_{e_1}\circ D_2-\ad_{e_2}\circ D_1 .
\]
In particular, \(\dim\BiDer(\mh)=2\dim\Der(\mh)=12\).
\end{prop}

\begin{proof}
For a linear map \(\beta\colon\mh\to\Der(\mh)\), consider
\[
\Phi_\beta\colon\mh\times\mh\longrightarrow\End(\mh),
\qquad
\Phi_\beta(y,z)=\beta([y,z])-\ad_y\circ\beta(z)+\ad_z\circ\beta(y),
\]
so that \(\beta\in\mathcal B(\mh)\) if and only if \(\Phi_\beta\) vanishes
identically. As \(\beta\), the bracket of \(\mh\) and the adjoint map are all
linear, \(\Phi_\beta\) is bilinear, and the skew-symmetry of the bracket gives
\(\Phi_\beta(z,y)=-\Phi_\beta(y,z)\). It is therefore enough to evaluate
\(\Phi_\beta\) on the pairs \((e_i,e_j)\) with \(i<j\), which leaves three
conditions. The pair \((e_1,e_2)\), together with \([e_1,e_2]=e_3\), gives
\begin{equation}\label{eq:heis_cocycle}
\beta(e_3)=\ad_{e_1}\circ\beta(e_2)-\ad_{e_2}\circ\beta(e_1),
\end{equation}
whereas the pairs \((e_1,e_3)\) and \((e_2,e_3)\), together with
\([e_i,e_3]=0\) and \(\ad_{e_3}=0\), give
\begin{equation}\label{eq:heis_cocycle_center}
\ad_{e_i}\circ\beta(e_3)=0,
\qquad i=1,2 .
\end{equation}

Let now \(\beta\in\mathcal B(\mh)\). Since \(\beta\) is determined by its values
on the basis and \Cref{eq:heis_cocycle} expresses \(\beta(e_3)\) in terms of
\(\beta(e_1)\) and \(\beta(e_2)\), the map
\(\beta\mapsto(\beta(e_1),\beta(e_2))\) is injective.

Conversely, take \(D_1,D_2\in\Der(\mh)\) and let \(\beta\) be the linear map
with \(\beta(e_1)=D_1\), \(\beta(e_2)=D_2\) and \(\beta(e_3)\) defined by
\Cref{eq:heis_cocycle}. We must check that \(\beta(e_3)\) is a derivation and
that relation \eqref{eq:heis_cocycle_center} holds. Write
\[
D_i=
\begin{pmatrix}
a_i & b_i & 0\\
c_i & d_i & 0\\
p_i & q_i & a_i+d_i
\end{pmatrix},
\qquad i=1,2 .
\]
As \(\ad_{e_1}=E_{32}\) and \(\ad_{e_2}=-E_{31}\), the matrix
\(\ad_{e_1}\circ D_2-\ad_{e_2}\circ D_1=E_{32}D_2+E_{31}D_1\) is obtained by
placing the second row of \(D_2\) and the first row of \(D_1\) in the third row.
All its entries vanish except those of the third row, which is
\[
(c_2+a_1,\ d_2+b_1,\ 0).
\]
In particular \(\beta(e_3)(e_3)=0\) and the image of \(\beta(e_3)\) is contained
in \(\F e_3=Z(\mh)\). A linear map with image in \(Z(\mh)\) satisfies
\([T(x),y]+[x,T(y)]=0\) for all \(x,y\in\mh\), hence it is a derivation exactly
when it vanishes on \([\mh,\mh]=\F e_3\). Both conditions hold, so
\(\beta(e_3)\in\Der(\mh)\). Finally, \(\ad_{e_i}\) vanishes on the centre of
\(\mh\), and the image of \(\beta(e_3)\) lies in \(Z(\mh)\), which is
\eqref{eq:heis_cocycle_center}.

The two assignments are inverse to each other by construction. Since
\(\dim\Der(\mh)=6\), the dimension count follows from
\Cref{prop:classical_biderivations_linear_sector}.
\end{proof}

\begin{remark}\label{rem:heisenberg_example_check}
The biderivations \(B_1,B_2\) of \Cref{ex:heisenberg_biderright} illustrate
\Cref{prop:bider_heisenberg}: their associated cocycles \(\beta_1,\beta_2\)
satisfy \Cref{eq:heis_cocycle}, as one checks directly on the matrices
\(\beta_i(e_j)=B_i(\cdot,e_j)\). It is also instructive to compare the
\(12\)-dimensional space \(\BiDer(\mh)\) with the matricial computations of
\cite{DIBARTOLOLAROSA2025}.
\end{remark}

\begin{remark}\label{rem:heisenberg_coboundaries}
In contrast with the semisimple case, most classical biderivations of \(\mh\)
are not of the inner type \eqref{eq:BT}. Indeed, for
\(y=y_1e_1+y_2e_2+y_3e_3\) one has \(\ad_y=y_1E_{32}-y_2E_{31}\), so
\(\ad_y\circ T\) is the matrix whose only nonzero row is the third, equal to
\(y_1(T_{21},T_{22},T_{23})-y_2(T_{11},T_{12},T_{13})\); this matrix is a
derivation of \(\mh\) if and only if its \((3,3)\) entry
\(y_1T_{23}-y_2T_{13}\) vanishes. Therefore
\[
\Tadm(\mh)
=
\{T\in\End(\mh)\mid T_{13}=T_{23}=0\}
=
\{T\in\End(\mh)\mid T(Z(\mh))\subseteq Z(\mh)\},
\]
which has dimension \(7\). Since \(\Hom(\mh,Z(\mh))\) has dimension \(3\), the
exact sequence of \Cref{thm:exact_sequence} shows that the space of
biderivations of the form \(B_T\) has dimension \(4\) inside the
\(12\)-dimensional space \(\BiDer(\mh)\): the map
\(\mathrm{cl}\colon\BiDer(\mh)\to H^1(\mh,\End(\mh))\) has an
\(8\)-dimensional image. This quantifies, in the smallest nilpotent example,
the failure of Whitehead-type vanishing outside the semisimple world.
\end{remark}

\section{Symmetric and Skew-Symmetric One-Sided Biderivations}
\label{sec:symm_and_skew-symm}

We now record some consequences of the relation between the right and left brackets under transposition. Throughout this section we keep the notation \(B^t(x,y)=B(y,x)\); recall from \Cref{prop:transpose_right_left} that \((B_1\rhd B_2)^t=B_1^t\lhd B_2^t\) for all \(B_1,B_2\in\BiDer_r(\mg)\).

\begin{lemma}\label{lem:rsymm_l}
Let \(B\in\BiDer_r(\mg)\). If \(B\) is either symmetric or skew-symmetric, then \(B\in\BiDer_l(\mg)\). In particular, \(B\) is a classical biderivation.
\end{lemma}

\begin{proof}
Assume first that \(B\) is symmetric. Since \(B(\cdot,y)\) is linear for every \(y\), for \(x,y_1,y_2\in\mg\) and \(\lambda_1,\lambda_2\in\F\) we get
\[
B(x,\lambda_1y_1+\lambda_2y_2)
=
B(\lambda_1y_1+\lambda_2y_2,x)
=
\lambda_1B(y_1,x)+\lambda_2B(y_2,x)
=
\lambda_1B(x,y_1)+\lambda_2B(x,y_2),
\]
so \(B\) is linear in the second variable. Moreover,
\begin{align*}
B(z,[x,y])
&=B([x,y],z)
=[x,B(y,z)]+[B(x,z),y]
=[x,B(z,y)]+[B(z,x),y],
\end{align*}
which is the left biderivation identity. If \(B\) is skew-symmetric, the same computations go through with the obvious sign changes: linearity in the second variable follows as above, and
\begin{align*}
B(z,[x,y])
&=-B([x,y],z)
=-[x,B(y,z)]-[B(x,z),y]
=[x,B(z,y)]+[B(z,x),y].
\end{align*}
In both cases \(B\in\BiDer_l(\mg)\).
\end{proof}

\begin{corollary}\label{cor:leftrightsymm_bider}
Every symmetric or skew-symmetric right biderivation is a classical biderivation. By transposition, the same holds for left biderivations.
\end{corollary}

By \Cref{lem:rsymm_l}, all the brackets appearing below are well-defined both as right and as left brackets.

\begin{prop}\label{prop:same_parity_brackets}
Let \(B_1,B_2\in\BiDer_r(\mg)\) be either both symmetric or both skew-symmetric. Then
\[
(B_1\rhd B_2)(x,y)=(B_1\lhd B_2)(y,x)
\]
for all \(x,y\in\mg\).
\end{prop}

\begin{proof}
If \(B_1,B_2\) are both symmetric, then \(B_i^t=B_i\) and \Cref{prop:transpose_right_left} gives \((B_1\rhd B_2)^t=B_1\lhd B_2\). If they are both skew-symmetric, then \(B_i^t=-B_i\), and by bilinearity of \(\lhd\),
\[
(B_1\rhd B_2)^t=(-B_1)\lhd(-B_2)=B_1\lhd B_2.
\]
In both cases the claim follows.
\end{proof}

\begin{prop}\label{prop:mixed_parity_brackets}
Let \(B_1,B_2\in\BiDer_r(\mg)\), and suppose that one of them is symmetric and the other is skew-symmetric. Then
\[
(B_1\rhd B_2)(x,y)=(B_2\lhd B_1)(y,x)
\]
for all \(x,y\in\mg\).
\end{prop}

\begin{proof}
Assume that \(B_1\) is symmetric and \(B_2\) is skew-symmetric, so that \(B_1^t=B_1\) and \(B_2^t=-B_2\). By \Cref{prop:transpose_right_left} and the skew-symmetry of \(\lhd\),
\[
(B_1\rhd B_2)^t
=
B_1\lhd(-B_2)
=
-(B_1\lhd B_2)
=
B_2\lhd B_1.
\]
The other case is analogous.
\end{proof}

\section{Examples}
\label{sec:examples}

We now describe some basic classes of examples, in order to make the previous constructions explicit and to show how the structure of \(\BiDer_r(\mg)\) depends on the derivation algebra \(\Der(\mg)\).

\subsection{The abelian case}

Let \(\mg\) be an abelian Lie algebra. Since the bracket of \(\mg\) vanishes identically, every linear endomorphism of \(\mg\) is a derivation, so \(\Der(\mg)=\End(\mg)\) and
\[
\BiDer_r(\mg)
\simeq
\mathcal F(\mg,\End(\mg)),
\]
with \(\rhd\) corresponding to the pointwise commutator in \(\End(\mg)\). If \(\mg\) is finite-dimensional, this can also be written as \(\mathcal F(\mg,\F)\otimes \End(\mg)\).

Moreover, in the abelian case both biderivation identities reduce to \(0=0\), so \emph{every} bilinear map \(\mg\times\mg\to\mg\) is a classical biderivation:
\[
\BiDer(\mg)=\Hom(\mg\otimes\mg,\mg).
\]
This shows that the space of classical biderivations can be very large, even though the defining identities are formally restrictive in the non-abelian case.

\subsection{The semisimple case}

Assume now that \(\mg\) is finite-dimensional and semisimple over a field of characteristic zero. Then \(\mg\) is complete, so by \Cref{cor:complete_current_consequences}
\[
\BiDer_r(\mg)
\simeq
\mathcal F(\mg,\F)\otimes\mg,
\]
and \(\BiDer_r(\mg)\) is a perfect Lie algebra with trivial center. Explicitly, an elementary tensor \(f\otimes X\) corresponds to the right biderivation
\[
B_{f,X}(x,y)=f(y)[X,x],
\]
and the bracket is \(B_{f,X}\rhd B_{h,Y}=B_{fh,[X,Y]}\). Thus, in the semisimple case, right biderivations are precisely \(\mg\)-valued currents on \(\mg\), with the pointwise current algebra bracket. As already observed, even when \(\mg\) is simple, \(\BiDer_r(\mg)\) is not simple, since every ideal \(I\) of \(\mathcal F(\mg,\F)\) yields the ideal \(I\otimes\mg\).

\subsection{The three-dimensional Heisenberg Lie algebra}

Let \(\mh\) be the three-dimensional Heisenberg Lie algebra with basis \(\{e_1,e_2,e_3\}\) and non-zero bracket \([e_1,e_2]=e_3\), so that \(e_3\) spans the center of \(\mh\).

We first describe \(\Der(\mh)\). Let \(D\in\Der(\mh)\) and write
\[
D(e_1)=a e_1+c e_2+e e_3,
\qquad
D(e_2)=b e_1+d e_2+f e_3.
\]
Since \(e_3=[e_1,e_2]\) is central, the derivation rule gives
\[
D(e_3)=[D(e_1),e_2]+[e_1,D(e_2)]=(a+d)e_3.
\]
Therefore, with respect to the ordered basis \(\{e_1,e_2,e_3\}\),
\[
\Der(\mh)
=
\left\{
\begin{pmatrix}
a & b & 0\\
c & d & 0\\
e & f & a+d
\end{pmatrix}
\ \middle|\
a,b,c,d,e,f\in\F
\right\},
\]
and every matrix of this form is easily checked to define a derivation of \(\mh\).

Consequently, a right biderivation \(B\in\BiDer_r(\mh)\) is the same thing as a family of matrices of the above form, depending on the parameter \(y\in\mh\): there exist functions \(a,b,c,d,e,f\colon \mh\to\F\) such that, for every \(y\in\mh\),
\begin{align*}
B(e_1,y)&=a(y)e_1+c(y)e_2+e(y)e_3,\\
B(e_2,y)&=b(y)e_1+d(y)e_2+f(y)e_3,\\
B(e_3,y)&=\bigl(a(y)+d(y)\bigr)e_3,
\end{align*}
and, conversely, any choice of six such functions defines a right biderivation of \(\mh\) by the formulas above and linearity in the first variable. Restricting to bilinear (resp.\ polynomial) right biderivations amounts to requiring the six coefficient functions to be linear (resp.\ polynomial) functions on \(\mh\).

This description makes the current algebra viewpoint concrete: a right biderivation of \(\mh\) is not a single derivation, but a parameter-dependent derivation \(y\mapsto B(\cdot,y)\in\Der(\mh)\), and the bracket \(\rhd\) is the commutator of these derivations at each value of the parameter.

\section{Pointwise Integration of Right Biderivations}
\label{sec:integration_of_rightbider}

Within the framework of Lie theory, differentiation and integration describe the passage between Lie groups and their Lie algebras. In the present setting, a right biderivation \(B\in\BiDer_r(\mg)\) assigns to every \(y\in\mg\) a derivation \(B(\cdot,y)\in\Der(\mg)\). The natural geometric interpretation, then, is not that \(B\) integrates to a single curve of automorphisms, but that it integrates pointwise to a family of such curves, one for each value of the parameter \(y\).

We begin with an elementary observation on maps which are linear in one argument. Let \(V\) be an \(n\)-dimensional vector space over a field \(\mathbb F\), and fix a basis \(\mathcal B_V=\{e_1,\dots,e_n\}\).

\begin{prop}\label{prop:linear_first_argument}
Let \(B\colon V\times V\to V\) be a map which is linear in the first argument. Then \(B\) is uniquely determined by the \(n\) maps \(f_i\colon V\to V\), \(1\leq i\leq n\), defined by \(f_i(y)=B(e_i,y)\).
\end{prop}

\begin{proof}
For every \(x=\sum_{i=1}^n x_i e_i\) and every \(y\in V\), linearity in the first argument gives
\[
B(x,y)=\sum_{i=1}^n x_i B(e_i,y)=\sum_{i=1}^n x_i f_i(y),
\]
so the maps \(f_i\) determine \(B\); conversely, any family of maps \(f_i\colon V\to V\) defines by this formula a map \(B\) which is linear in the first argument. Uniqueness follows by evaluating at \(x=e_i\).
\end{proof}

From now on, let \(g\colon V\to\mathbb F\) be a not necessarily linear map taking values in the base field, and let \(F\colon V\to V\) be a map. We use the notation
\begin{equation}\label{eq:right_bider_scalarXder}
    B_{g,F}(x,y)=g(y)F(x).
\end{equation}
If \(F\) is linear, then \(B_{g,F}\) is linear in the first argument. Note that, in this separated form, the maps \(f_i\) of \Cref{prop:linear_first_argument} are \(f_i(y)=g(y)F(e_i)\): they all depend on the same scalar function \(g\), and each takes values in the line generated by \(F(e_i)\).

When \(V=\mg\) is a Lie algebra, the separated maps above give right biderivations precisely when the linear component is a derivation.

\begin{thm}\label{thm:F_is_a_derivation}
Let \(g\colon\mg\to\F\) be a nonzero map, and let \(F\colon\mg\to\mg\) be a map. Then \(B_{g,F}\) is a right biderivation of \(\mg\) if and only if \(F\in\Der(\mg)\).
\end{thm}

\begin{proof}
Set \(B=B_{g,F}\). If \(F\in\Der(\mg)\), then for every \(x,y,z\in\mg\),
\begin{align*}
B([x,y],z)
&=g(z)F([x,y])
=g(z)[F(x),y]+g(z)[x,F(y)]
=[B(x,z),y]+[x,B(y,z)],
\end{align*}
so \(B\) is a right biderivation.

Conversely, assume that \(B\) is a right biderivation. Since \(g\not\equiv0\), there exists \(z_0\in\mg\) with \(g(z_0)\neq0\). The right biderivation identity applied at \(z_0\) gives
\[
g(z_0)F([x,y])
=
g(z_0)[x,F(y)]+g(z_0)[F(x),y]
\]
for all \(x,y\in\mg\), and dividing by \(g(z_0)\) yields the Leibniz rule for \(F\). Moreover, \(B(\cdot,z_0)=g(z_0)F\) is linear because \(B\) is a right biderivation, hence \(F\) is linear. Therefore \(F\in\Der(\mg)\).
\end{proof}

\begin{prop}\label{prop:rank_one_closed}
Let \(B_i=B_{g_i,F_i}\), \(i=1,2\), with \(g_i\colon\mg\to\F\) and \(F_i\in\Der(\mg)\). Then
\[
B_1\rhd B_2=B_{g_1g_2,[F_1,F_2]}.
\]
\end{prop}

\begin{proof}
For every \(x,y\in\mg\),
\begin{align*}
(B_1\rhd B_2)(x,y)
&=B_1(g_2(y)F_2(x),y)-B_2(g_1(y)F_1(x),y)\\
&=g_1(y)g_2(y)\bigl(F_1F_2(x)-F_2F_1(x)\bigr)
=g_1(y)g_2(y)[F_1,F_2](x),
\end{align*}
and \([F_1,F_2]\in\Der(\mg)\).
\end{proof}

\begin{remark}
The family of maps \(B_{g,F}\) is closed under the bracket \(\rhd\), but it is not, in general, a vector subspace of \(\BiDer_r(\mg)\): a sum \(g_1(y)F_1(x)+g_2(y)F_2(x)\) need not be expressible in the form \(g(y)F(x)\).
\end{remark}

We now pass to the geometric part of the section. From this point on, Lie algebras and Lie groups are assumed to be defined over \(\mathbb R\).

Let \(G\) and \(H\) be matrix Lie groups, with Lie algebras \(\mg\) and \(\mathfrak h\), respectively. If \(\Phi\colon G\to H\) is a Lie group homomorphism, then there exists a unique \(\mathbb R\)-linear map \(\phi\colon\mg\to\mathfrak h\) such that \(\Phi(e^x)=e^{\phi(x)}\) for all \(x\in\mg\), and by Theorem 3.28 in \cite{hall2015lie},
\[
\phi(x)=\left.\frac{d}{dt}\right|_{t=0}\Phi(e^{tx}).
\]
The same statement holds for general Lie groups, with the matrix exponential replaced by the exponential map \(\exp\colon\mg\to G\).

Let now \(G\) be a Lie group with identity element \(e\), and let \(\mg=\Lie(G)\). If \(\psi\colon G\to G\) is an automorphism of \(G\), then \(d\psi_e\colon\mg\to\mg\) is an automorphism of Lie algebras (Theorem 8.44 in \cite{lee2012introduction}; note that \(\psi(e)=e\), so \(d\psi_e\) is indeed an endomorphism of \(T_eG=\mg\)), and for every \(x\in\mg\),
\[
d\psi_e(x)=\left.\frac{d}{dt}\right|_{t=0}\psi(\exp(tx)).
\]

If \(G\) is simply connected, then \(\Aut(G)\) is a Lie group whose Lie algebra is naturally identified with \(\Der(\mg)\) (Proposition 2, Chapter IV, in \cite{Chevalley1946}). Therefore, every derivation of \(\mg\) can be realised as the infinitesimal generator of a smooth curve of automorphisms of \(G\): if \((\varphi_s)_{s\in\mathbb R}\subseteq\Aut(G)\) is a smooth curve with \(\varphi_0=\id_G\), then
\[
D=\left.\frac{d}{ds}\right|_{s=0}(d\varphi_s)_e
\]
is a derivation of \(\mg\) and, conversely, every \(D\in\Der(\mg)\) arises in this way.

Applying this to right biderivations, we obtain the pointwise integration of the identification \(\BiDer_r(\mg)\simeq \mathcal F(\mg,\Der(\mg))\).

\begin{thm}\label{thm:pointwise_integration}
Let \(G\) be a simply connected Lie group with identity element \(e\) and \(\Lie(G)=\mg\), and let \(B\in\BiDer_r(\mg)\). Then, for every \(y\in\mg\), there exists a smooth curve of automorphisms
\[
(\varphi_s^y)_{s\in\mathbb R}\subseteq\Aut(G),
\qquad
\varphi_0^y=\id_G,
\]
such that
\[
B(x,y)=
\left.\frac{d}{ds}\right|_{s=0}
(d\varphi_s^y)_e(x)
\]
for every \(x\in\mg\).
\end{thm}

\begin{proof}
Fix \(y\in\mg\). Since \(B\in\BiDer_r(\mg)\), the map \(D_y=B(\cdot,y)\) is a derivation of \(\mg\), and since \(G\) is simply connected, \(D_y\) integrates to a smooth curve of automorphisms \((\varphi_s^y)_{s\in\mathbb R}\subseteq\Aut(G)\) with \(\varphi_0^y=\id_G\) and
\[
D_y=
\left.\frac{d}{ds}\right|_{s=0}
(d\varphi_s^y)_e.
\]
Evaluating at \(x\in\mg\) gives the claim.
\end{proof}

\begin{remark}
No regularity in the parameter \(y\) is claimed in \Cref{thm:pointwise_integration}: the theorem only says that each derivation \(B(\cdot,y)\) can be integrated separately. If one restricts to those right biderivations for which \(y\mapsto B(\cdot,y)\) is smooth as a map \(\mg\to\Der(\mg)\), the identification \(\BiDer_r(\mg)\simeq\mathcal F(\mg,\Der(\mg))\) suggests the usual current-type picture: such biderivations should be regarded as the Lie algebra of a suitable mapping group with values in \(\Aut(G)\), with all brackets taken pointwise.
\end{remark}

The separated class \(B_{g,F}\) admits an especially explicit integration: the curve integrating \(B_{g,F}(\cdot,y)\) is obtained by rescaling the time parameter by \(g(y)\).

\begin{corollary}\label{cor:separated_integration}
Let \(G\) be a simply connected Lie group with Lie algebra \(\mg\), and let \(B_{g,F}\in\BiDer_r(\mg)\), with \(g\colon\mg\to\mathbb R\) and \(F\in\Der(\mg)\). Let \((\varphi_t)_{t\in\mathbb R}\subseteq\Aut(G)\) be a smooth one-parameter group of automorphisms generated by \(F\), and, for every \(y\in\mg\), define
\[
\varphi_s^y=\varphi_{s g(y)}.
\]
Then \((\varphi_s^y)_{s\in\mathbb R}\) is a smooth curve of automorphisms with \(\varphi_0^y=\id_G\), and
\[
B_{g,F}(x,y)
=
\left.\frac{d}{ds}\right|_{s=0}
(d\varphi_s^y)_e(x)
\]
for every \(x,y\in\mg\).
\end{corollary}

\begin{proof}
For fixed \(y\in\mg\), the curve \(s\mapsto \varphi_{s g(y)}\) is a smooth curve in \(\Aut(G)\) with \(\varphi_0^y=\id_G\), and by the chain rule
\[
\left.\frac{d}{ds}\right|_{s=0}
(d\varphi_{s g(y)})_e(x)
=
g(y)
\left.\frac{d}{dt}\right|_{t=0}
(d\varphi_t)_e(x)
=g(y)F(x)
=B_{g,F}(x,y).
\]
\end{proof}

\begin{remark}
The last corollary should be read as a time-rescaling construction: we are not multiplying a Lie group automorphism by the scalar \(g(y)\); rather, for each fixed \(y\), we take the one-parameter group generated by \(F\) and evaluate it at the rescaled time \(s g(y)\).
\end{remark}

\section*{Acknowledgments}

Alfonso Di Bartolo was supported by the University of Palermo (FFR2024, UNIPA BsD). Gianmarco La Rosa was supported by "Sustainability Decision Framework (SDF)" Research Project --  CUP B79J23000540005 -- Grant Assignment Decree No. 5486 adopted on 2023-08-04. The authors are members of the “National Group for Algebraic and Geometric Structures, and Their Applications” (GNSAGA — INdAM). 

\section*{Conflict of Interest}

The authors have no conflict of interest to declare that is relevant to this article.

\section*{Data Availability}
Data sharing is not applicable to this article as no datasets were generated or analysed during the current study.


\printbibliography
\end{document}